\documentclass{article} %amsart} %
\usepackage{amsfonts}
\usepackage{amsmath}
\usepackage{theorem}
\usepackage{array}
\usepackage[a4paper]{geometry}
\usepackage{amssymb}
\usepackage{graphics}
\usepackage[usenames]{color}
\usepackage{graphicx,url}
\usepackage{latexsym, euscript, mathrsfs}
\usepackage{etoolbox}
\usepackage{hyperref}

\makeatletter
\newcounter{author}
\renewcommand*\author[1]{%
  \stepcounter{author}%
  \ifnum\c@author=1
    \gdef\@author{#1}%
  \else
    \xdef\@author{\unexpanded\expandafter{\@author\and#1}}%
  \fi
  \csgdef{author@\the\c@author}{#1}}
\newcommand*\email[1]{%
  \csgdef{email@\the\c@author}{#1}}
\newcommand*\address[1]{%
  \csgdef{address@\the\c@author}{#1}}
\AtEndDocument{%
  \xdef\author@count{\the\c@author}%
  \c@author=1
  \print@authors}
\newcommand*\print@authors{%
  \ifnum\c@author>\author@count
  \else
    \print@author{\the\c@author}%
    \advance\c@author by 1
    \expandafter\print@authors
  \fi}
\newcommand*\print@author[1]{%
  \par\medskip
  \begin{tabular}{@{}l@{}}%
    \csuse{address@#1}\\
    \textit{E-mail address}:
    \href{mailto:\csuse{email@#1}}{\csuse{email@#1}}
  \end{tabular}}
\makeatother

\newtheorem{theorem}{Theorem}[section]

\newtheorem{conjecture}[theorem]{Conjecture}
\newtheorem{corollary}[theorem]{Corollary}

\newtheorem{definition}[theorem]{Definition}
\newtheorem{remark}[theorem]{Remark}

\newtheorem{lemma}[theorem]{Lemma}

\newtheorem{proposition}[theorem]{Proposition}

\newenvironment{proof}[1][Proof]{\textbf{#1.} }{\hfill $\Box$ \medskip} % \rule{0.5em}{0.5em}}
\newtheorem{example}[theorem]{Example}
\newtheorem{fact}[theorem]{Fact}

\def\Aff{{\mathbb A}}
\def\PP{{\mathbb P}}
\def\Plm{{\mathbb P}^{\ell m-1}}
\def\F{{\mathbb F}}
\def\Fq{{\mathbb F}_q}

\def\Mat{{\mathbb M}}
\def\Matlm{{\Mat}_{\ell\times m}}
\def\Det{{\mathcal D}}
\def\Dt{{\Det}_t}

\def\Dtp{\widehat{\Det}_t}

\def\It{{\mathcal{I}}_{t+1}}
\def\Ct{C_{\det}(t;\ell,m)}
\def\Cpt{\widehat{C}_{\det}(t;\ell,m)}

\def\Cpone{\widehat{C}_{\det}(1;\ell,m)}
\newcommand{\wH}{{\mathrm{w_H}}}
\newcommand{\rk}{\operatorname{rank}}

\newcommand{\Ev}{\operatorname{Ev}}

\begin{document}
\title{Hyperplane Sections of Determinantal Varieties \\ over Finite Fields and Linear Codes}
%\author{Peter Beelen and Sudhir R. Ghorpade}
\author{Peter Beelen\thanks{Partially supported by the Danish Council for Independent Research (Grant No. DFF--4002-00367).}}
%\\ Department of Applied Mathematics and Computer Science / 
%Technical University of Denmark / 2800 Kgs. Lyngby, Denmark \\
%\and
%Sudhir R. Ghorpade \\
%Department of Mathematics, 
%Indian Institute of Technology Bombay, %\newline \indent
%Powai, Mumbai 400076, India}
\address{\sc Department of Applied Mathematics and Computer Science, \\ %\newline \indent
\sc Technical University of Denmark,  %\newline \indent Asmussens All{\'e}, Building 303B, Room 150, \newline \indent
2800 Kgs. Lyngby, Denmark}
\email{\tt pabe@dtu.dk}
  %  author two information
\author{Sudhir R. Ghorpade\thanks{Partially supported by  IRCC Award grant 12IRAWD009 from IIT Bombay.}}
%\address
\address{\sc Department of Mathematics,
Indian Institute of Technology Bombay,  \\%\newline \indent
\sc Powai, Mumbai 400076, India.}
\email{\tt srg@math.iitb.ac.in}
\date{\empty}
\maketitle

\begin{abstract} We determine the number of $\Fq$-rational points of hyperplane sections of classical determinantal varieties %consider a class of linear codes associated to projective algebraic varieties
defined by the vanishing of minors of a fixed size of a generic matrix, and identify sections giving the maximum number of $\Fq$-rational points. Further we consider similar questions for sections by linear subvarieties of a fixed codimension in the ambient projective space.  This is closely related to the  study of linear codes associated to determinantal varieties, and the determination of their weight distribution, minimum distance and generalized Hamming weights. The previously known results about these 
% known results
%Results in \cite{detcodes1}
%concerning the minimum distance and generalized Hamming weights of such codes 
are generalized and expanded significantly.
\end{abstract}

%%--The name(s) of the author(s), their e-mails and addresses should be placed here
%\textsc{Peter Beelen} \hfill \texttt{pabe@dtu.dk} \\
%{\small Department of Applied Mathematics and Computer Science,  \newline
%Technical University of Denmark, DK 2800, Kgs. Lyngby, Denmark} \\ \\
%\textsc{Sudhir R. Ghorpade} \hfill \texttt{srg@math.iitb.ac.in} \\
%{\small Department of Mathematics,
%Indian Institute of Technology Bombay, \newline
%Powai, Mumbai 400076, India}

%\keywords{linear codes, determinantal varieties, generalized Hamming weight, weight distribution}

\section{Introduction}

The classical determinantal variety %$\Dtp(\ell,m)$
defined by the vanishing all minors of a fixed size in a generic matrix is an object of considerable importance and ubiquity in algebra, combinatorics, algebraic geometry, invariant theory and representation theory. The defining equations clearly have integer coefficients and as such the variety can be defined over any finite field. The number of $\Fq$-rational points of this variety is classically known. We are mainly interested in a more challenging question of determining the number of $\Fq$-rational points of such a variety when intersected with a hyperplane in the ambient projective space, or more generally, with a linear subvariety of a fixed codimension in the ambient projective space. In particular, we wish to know which of these sections have the maximum number of $\Fq$-rational points.  These questions are directly related to determining the complete weight distribution and the generalized Hamming weights of the associated linear codes, which are caledl \emph{determinantal codes}. In this setting, the problem was considered in \cite{detcodes1} and a beginning was made by showing that the determination of the weight distribution is related to the problem of computing the number of generic matrices of a given rank with a nonzero ``partial trace''. More definitive % concrete 
results were obtained in the special case of  varieties defined by the vanishing of all $2\times 2$ minors of a generic matrix. Here we settle the question of determination of the weight distribution and  the minimum distance of %the %associated linear 
determinantal codes  in complete generality. Further, we %and also 
determine some initial and terminal %the first and last few 
generalized Hamming weights of determinantal codes. 
We also show that the determinantal codes have a very low dual minimum distance (viz., 3), which makes them rather interesting from the point of view of coding theory. Analogous problems have been considered for other classical projective varieties such as Grassmannians, Schubert varieties, etc., leading to interesting classes of linear codes which have been of some current interest; see, for example, \cite{N}, \cite{GT}, \cite{HJR2}, \cite{X}, \cite{GS}, and  the survey \cite{Little}.
%It may be noted that 
%In geometric terms, the determination of the minimum distance corresponds to the determination of the 
%maximum number of $\Fq$-rational points on hyperplane sections, whereas the determination of the $r$-th higher weight corresponds to finding the maximum number of $\Fq$-rational points on sections by projective linear subspaces of codimension $r$. 
%; see, e.g., %for example,  \cite{TV1}. 

As was mentioned in \cite{detcodes1} and further explained in the next section and Remark \ref{rem:Delsarte}, the results on the weight distribution of determinantal codes are also related to the work of Delsarte \cite{D} on eigenvalues of association schemes of bilinear forms using the rank metric as distance. We remark also that a special case of these results has been looked at by Buckhiester \cite{buckhiester}. 

A more detailed description of the contents of this paper is given in the next section, while the main results
%of the paper
are proved in the two %subsequent
sections that follow the next section. An  appendix contains self-contained and alternative proofs of some results that were deduced from the work of Delsarte and this might be of an independent interest. % as well.

%
%
%
%A useful and interesting way to construct a linear code is to consider a projective algebraic variety $V$ defined over the finite field $\Fq$ with $q$ elements together with a nondegenerate  embedding in a projective space, and to look at the projective system (in the sense of Tsfasman and Vl\u{a}du\c{t} \cite{TV1}) associated to the $\Fq$-rational points of $V$. A good illustration is provided by the case of Grassmann codes and Schubert codes, which have been of much interest; see, for example, \cite{N}, \cite{GT}, \cite{HJR2}, \cite{X} or the survey \cite{Little}.
%
%In this paper we consider a class of linear codes that are associated to classical determinantal varieties. These will be referred to as determinantal codes and were introduced in \cite{detcodes1}. The present work should be seen as a continuation of \cite{detcodes1}, in which the length and dimension of these codes were determined. In \cite{detcodes1} all possible weights of codewords in a determinantal code were determined, but since these weights are given by rather complex expressions, it is not a priori clear what the minimum distance of a determinantal code is. In \cite{detcodes1} this was determined for a special case, but here we settle the question in complete generality. Finally, we determine the first and last few generalized Hamming weights of determinantal codes.

\section{Preliminaries}
\label{sec:prelim}

Fix throughout this paper a prime power $q$, positive integers $t, \ell,m$, and an $\ell\times m$ matrix $X = (X_{ij})$ whose entries are independent indeterminates over $\Fq$. We will denote by $\Fq[X]$ the polynomial ring in the $\ell m$ variables $X_{ij}$ ($1\le i \le \ell$, $1\le j \le m$) with coefficients in $\Fq$. As usual, by a \emph{minor} of size $t$ or a $t\times t$ minor of $X$ we mean the determinant of a $t\times t$ submatrix of $X$, where $t$ is a nonnegative integer $\le \min\{\ell, m\}$. As per standard conventions, the only  $0\times 0$ minor of $X$ is $1$. We will be mostly interested in the class of minors of a fixed size, and this class is unchanged if $X$ is replaced by its transpose. With this in view, we shall always assume, without loss of generality, that $\ell\le m$. Given a field $\F$, we denote by $\Matlm (\F)$ the set of all $\ell\times m$ matrices with entries in $\F$. Often $\F=\Fq$ and in this case we may simply write
$\Matlm$ for $\Matlm (\Fq)$. Note that $\Matlm $ can be viewed as an affine space $\Aff^{\ell m}$ over $\Fq$ of dimension $\ell m$. For $0\le t \le \ell$, the corresponding classical determinantal variety (over $\Fq$) is denoted by $\Dt$ and defined as  the affine algebraic variety in $\Aff^{\ell m}$ given by the vanishing of all $(t+1) \times (t+1)$ minors of $X$; in other words
$$
\Dt(\ell,m) = \left\{ M\in \Matlm (\Fq) : \rk (M) \le t \right\}.
$$
Note that $\Det_0$ only consists of the zero-matrix. For $t=\ell$, no $(t+1)\times(t+1)$ minors of $X$ exist. This means that $\Det_\ell=\Matlm$, which is in agreement with the above description of $\Det_\ell$ as the set of all matrices of rank at most $\ell$.

It will also be convenient to define the sets
$$
\mathfrak{D}_t(\ell,m):=\left\{ M\in \Matlm (\Fq) : \rk (M) = t \right\},
$$
for $0 \le t \le \ell$ as well as their cardinalities $\mu_t(\ell,m):=|\mathfrak{D}_t(\ell,m)|$. The map that sends $M \in \mathfrak{D}_t(\ell,m)$ to its row-space is a surjection of
%the space of all $\ell\times m$ matrices of rank $j$
$\mathfrak{D}_t(\ell,m)$ onto the space $G_{t,m}$ of $t$-dimensional subspaces of $\Fq^m$. Moreover for a given  $W\in  G_{t,m}$, the number of $M\in \mathfrak{D}_t(\ell,m)$ with row-space $W$ is equal to the number of $\ell\times t$ matrices over $\Fq$ of rank $t$ or equivalently, the number of $t$-tuples of linearly independent %(column)
vectors in $\Fq^{\ell}$. Since $|G_{j,m}|$ is %given by
the Gaussian binomial coefficient ${{m} \brack{j}}_q$,
we find %it follows that
\begin{equation}
\label{mur}
{\mu}_t(\ell,m) =
|\mathfrak{D}_t(\ell,m)| = {{m} \brack{t}}_q \prod_{i=0}^{t-1} (q^{\ell} -q^i) = %q^{\frac{j(j-1)}{2}}
q^{\binom{t}{2}} \prod_{i=0}^{j-1} \frac{ \left( q^{\ell-i}-1\right) \left( q^{m-i}-1\right) }{q^{i+1}-1}.
\end{equation}
Using the Gaussian factorial $[n]_q!:=\prod_{i=1}^n(q^i-1)$, one can also give the following alternative expressions:
$${\mu}_t(\ell,m) =q^{\binom{t}{2}}{{m} \brack{t}}_q \dfrac{[\ell]_q!}{[\ell-t]_q!}=q^{\binom{t}{2}}\dfrac{[m]_q![\ell]_q!}{[m-t]_q![t]_q![\ell-t]_q!}=q^{\binom{t}{2}}{{\ell} \brack{t}}_q \dfrac{[m]_q!}{[m-t]_q!}.$$
Note that $\mu_0(\ell,m)=1$. Next we define $$\nu_t(\ell,m):=\sum_{s=0}^t \mu_s(\ell,m).$$
Since $\Dt(\ell,m)$ is the disjoint union of $\mathfrak D_0(\ell,m),\dots,\mathfrak D_t(\ell,m)$, we have
\begin{equation}
\label{mur2}
\nu_t(\ell,m)=\sum_{s=0}^t \mu_s(\ell,m)=|\Dt(\ell,m)|.
\end{equation}
The affine variety $\Dt(\ell,m)$ is, in fact, a cone; in other words, the vanishing ideal $\It$ (which is precisely the ideal of $\Fq[X]$ generated by all $(t+1) \times (t+1)$ minors of $X$) is a homogeneous ideal.
Also it is a classical (and nontrivial) fact that $\It$ is a prime ideal (see, e.g., \cite{survey}).
Thus $\Dt(\ell,m)$ can also be
viewed as a projective algebraic variety in $\Plm$, and viewed this way, we will denote it by $\Dtp(\ell,m)$. We remark that the dimension of $\Dtp(\ell,m)$ is $t(\ell + m -t) -1$ (cf. \cite{survey}). Moreover
\begin{eqnarray*}
|\Dtp(\ell,m)| & = & \dfrac{|\Dt(\ell,m)|-1}{q-1}=\dfrac{1}{q-1}\sum_{s=1}^t \mu_s(\ell,m).%\\
% & = &  \sum_{s=1}^t \hat{\mu}_s(\ell,m), \ \makebox{with} \ \hat{\mu}_s(\ell,m):=\dfrac{\mu_s(\ell,m)}{q-1}.
\end{eqnarray*}

We are now ready to define the codes we wish to study. Briefly put, the determinantal code $\Cpt$ is the linear code corresponding to the projective system
$\Dtp(\ell,m) \hookrightarrow \Plm (\Fq) = \PP(\Matlm)$. An essentially equivalent way to obtain this code is as follows: Denote $\hat{n}=|\Dtp(\ell,m)|$ and choose an ordering $P_1,\cdots, P_{\hat{n}}$ of the elements of $\Dtp(\ell,m)$. Further choose representatives $M_i \in \Matlm$ for $P_i$. Then consider the evaluation map
\begin{equation*}
\label{EvMap}
\Ev: \Fq[X]_1 \to \Fq^{\hat{n}} \quad \text{defined by} \quad \Ev(f) = \hat{c}_f:=  \left(f(M_1), \dots , f(M_{\hat{n}}) \right),
\end{equation*}
where $\Fq[X]_1$ denotes the space of homogeneous polynomials in $\Fq[X]$ of degree $1$ together with the zero polynomial. The image of this evaluation map can directly be identified with $\Cpt$. A different choice of representatives or a different ordering of these representatives gives in general a different code, but basic quantities like minimum distance, weight distribution, and generalized Hamming weights are independent on these choices.

In \cite{detcodes1}, also another code $\Ct$ was introduced. It can be obtained by evaluating functions in $\Fq[X]_1$ in all elements of $\Dt(\ell,m)$. The parameters of $\Ct$ determine those of $\Cpt$ and vice-versa, see \cite[Prop. 1]{detcodes1}. It is therefore sufficient to study either one of these two codes. In the remainder of this article, we will focus on $\Cpt$ and determine some of its basic parameters. It is in a sense also a more natural code to study, since the code $\Ct$ is degenerate, whereas $\Cpt$ is nondegenerate \cite{detcodes1}. We quote this and some other useful facts from \cite[Prop.1, Lem. 1, Cor. 1]{detcodes1}:

\begin{fact}
\label{fact:fromdetcodes1}
The code
$\Cpt$ is a nondegenerate code of dimension $\hat{k} =\ell m$ and length $\hat{n} = |\Dtp(\ell,m)|.$ For $f=\sum_{i=1}^\ell\sum_{j=1}^m f_{ij}X_{ij} \in \Fq[X]_1$, denote by $F=(f_{ij})$ the coefficient matrix of $f$. Then the Hamming weight of the corresponding
codeword $\hat{c}_f$ of $\Cpt$ depends only on $\rk(F)$. Consequently, if $r= \rk(F)$, then
$$\wH(\hat{c}_f) = \wH(\hat{c}_{\tau_r}), \quad \mbox{where} \quad \tau_r:= X_{11} + \cdots + X_{rr}.$$
As a result, the code $\Cpt$ has at most $\ell+1$ distinct weights, $\hat{w}_0(t;\ell,m),\dots, \hat{w}_{\ell}(t;\ell,m)$, given by %$w_0=0=\hat{w}_0$ while
$\hat{w}_r(t;\ell,m) = \wH(\hat{c}_{\tau_r})$ for $r=0,1,\dots , \ell$.
%Moreover, the weight enumerator polynomial $ \hat{A} (Z)$ of $\Cpt$ is given by
%$$
%\hat{A}(Z) =   (q-1)\sum_{r=0}^{\ell} \hat{\mu}_r(\ell, m) Z^{\hat{w}_r}.
%$$
\end{fact}

We call the function $\tau_r$ the $r^{\rm th}$ partial trace. Note that $\hat{w}_0(t;\ell,m)=0$, since $\tau_0=0$. To determine the other weights $\hat{w}_r(t;\ell,m)$, one would need to count the number of $M \in \Matlm(\Fq)$ of rank at most $t$ with nonzero $r$-th partial trace. Delsarte \cite{D} used the theory of association schemes to solve an essentially equivalent problem of determining the number $\mathfrak{w}_r(t;\ell,m)$ of $M \in \mathfrak{D}_t(\ell,m)$ with $\tau_r(M)\ne 0$, and showed: % that
%N_t(r) = \frac{ (q-1 )\left( \mu_t - P_t(r) \right)}{q} \quad \text{where} \quad
%P_t(r) = \sum_{i=0}^{\ell} (-1)^{t-i} q^{im + \binom{k-i}{2}} {{m-i}\brack{m-t}}_q {{m-r}\brack{i}}_q.
\begin{equation*}
\mathfrak{w}_r(t;\ell,m) = \frac{ q-1 }{q} \left( \mu_t (\ell,m) -
 \sum_{i=0}^{\ell} (-1)^{t-i} q^{im + \binom{t - i}{2}} {{\ell -i}\brack{\ell -t}}_q {{\ell -r}\brack{i}}_q  \right).
\end{equation*}
The case $r=\ell=m$ was already dealt with by Buckhiester in \cite{buckhiester}. In the appendix of this paper, we obtain using different methods an alternative formula for $\mathfrak{w}_r(t;\ell,m)$, which may be of independent interest. For future use, we define $\hat{\mathfrak{w}}_r(t;\ell,m):={\mathfrak{w}}_r(t;\ell,m)/(q-1)$ for $0 \le r \le \ell$ and $1 \le t \le \ell$. From Delsarte's result it follows that
\begin{equation}\label{eq:delsarte}
\hat{\mathfrak{w}}_r(t;\ell,m) = \frac{ 1 }{q} \left( \mu_t (\ell,m) -
 \sum_{i=0}^{\ell} (-1)^{t-i} q^{im + \binom{t - i}{2}} {{\ell -i}\brack{\ell -t}}_q {{\ell -r}\brack{i}}_q  \right),
\end{equation}

In fact, Delsarte considered codes obtained by evaluating elements from $\Fq[X]_1$ in all $\ell \times m$ matrices of rank $t$. Therefore his code can be seen as the projection of $\Ct$ on the coordinates corresponding to matrices of rank $t$.

Using Equation \eqref{eq:delsarte}, we see that the nonzero weights of $\Cpt$ are given by
\begin{eqnarray}\label{eq:weights}
\hat{w}_r(t;\ell,m) & = &  \sum_{s=1}^t \hat{\mathfrak{w}}_r(s;\ell,m)\\
 & = &\sum_{s=1}^t\frac{ 1 }{q} \left( \mu_t (\ell,m) -
 \sum_{i=0}^{\ell} (-1)^{t-i} q^{im + \binom{t - i}{2}} {{\ell -i}\brack{\ell -t}}_q {{\ell -r}\brack{i}}_q  \right),\notag
\end{eqnarray}
for $r=1, \dots , \ell$. However, for a fixed $t$, it is not obvious how $\hat{w}_1(t;\ell,m), \dots , \hat{w}_{\ell}(t;\ell,m)$ are ordered or even which among them is the least. We will formulate a conjecture based (among others) on the following examples.

\begin{example}\label{ex:t=1}
If $t=0$ the code $C_{\det}(t;\ell,m)$ is trivial (containing only the zero word), while the code $\widehat{C}_{\det}(t;\ell,m)$ is not defined. Therefore the easiest nontrivial case occurs for $t=1$. This case was considered in \cite{detcodes1}, where is was shown that $$\hat{w_r}(1;\ell,m)=q^{\ell + m -2} + q^{\ell + m-3} + \dots +  q^{\ell + m -r-1}=q^{\ell + m -r -1} \frac{(q^r-1)}{q-1}.$$ These formulae also follow fairly directly from Equations \eqref{eq:delsarte} and \eqref{eq:weights}. It follows directly that $\hat{w}_1(1;\ell,m) < \hat{w}_2(1;\ell,m) < \dots  < \hat{w}_{\ell}(1;\ell,m)$ and that $\hat{w}_1(1;\ell,m)=q^{\ell + m -2}$ is the minimum distance of $\Cpone$.
\end{example}

\begin{example}\label{ex:t45}
In this example we consider the determinantal code $\widehat{C}_{\det}(t;4,5)$ in case $q=2$ and $1 \le t \le 5$. %Note that for $q=2$ the codes $C_{\det}(t;\ell,m)$ and $\widehat{C}_{\det}(t;\ell,m)$ are identical for any value of $t$, $\ell$ and $m$.
Using the formulae in Equations \eqref{eq:delsarte} and \eqref{eq:weights}, we find the following table:
$$
\begin{array}{|c|c|c|c|c|}
\hline
r & 1& 2 & 3 & 4 \\
\hline
\hat{w}_r(1;\ell,m) & 128 & 192 & 224 & 240\\
%\hline
%\hat{w}_r(1;\ell,m)-\hat{\mu}_t(\ell,m) & -105 & -41 & -9 & 7\\
%\hline
\hline
\hat{w}_r(2;\ell,m) & 13568 & 16256 & 16576 & 16480\\
\hline
%\hat{w}_r(2;\ell,m)-\hat{\mu}_t(\ell,m) & -2940 & -252 & 68 & -28\\
%\hline
%\hline
\hat{w}_r(3;\ell,m) & 201728 & 212480 & 211712 & 211840\\
\hline
%\hat{w}_r(3;\ell,m)-\hat{\mu}_t(\ell,m) & -10080 & 672 & -96 & 32\\
%\hline
%\hline
\hat{w}_r(4;\ell,m) & 524288 & 524288 & 524288 & 524288\\
\hline
%\hat{w}_r(4;\ell,m)-\hat{\mu}_t(\ell,m) & 0 & 0 & 0 & 0\\
%\hline
\end{array}
$$

One sees that it is not true in general that $\hat{w}_r(t;\ell,m) < \hat{w}_s(t;\ell,m)$ whenever $r<s$. However, in this example it is true that for a given $t$, the weight $\hat{w}_1(t;\ell,m)$ is the smallest among all nonzero weights $\hat{w}_r(t;\ell,m)$.
\end{example}

\begin{example}\label{ex:t=ell}
In case $t=\ell$ in the previous example, all weights $\hat{w}_1,\dots,\hat{w}_{\ell}$ were the same. This holds in general: If $t=\ell$, then $\Dtp = \PP^{\ell m-1}$ and $\Cpt$ is a first order projective Reed-Muller code (cf. \cite{L}). All nonzero codewords in this code therefore have weight $q^{\ell m -1}$. Note that combining this with Equations \eqref{eq:delsarte} and \eqref{eq:weights} we obtain for $1 \le r \le \ell$ the following identity
$$q^{\ell m-1}=\sum_{s=1}^{\ell}\frac{ 1}{q} \left( \mu_s (\ell,m) -
 \sum_{i=0}^{\ell} (-1)^{s-i} q^{im + \binom{s - i}{2}} {{\ell -i}\brack{\ell -s}}_q {{\ell -r}\brack{i}}_q  \right).$$
Using Equation \eqref{mur2} with $t=\ell$, we see that for $1 \le r \le \ell$ apparently the following identity holds
$$\sum_{s=1}^\ell\sum_{i=0}^{\ell} (-1)^{s-i} q^{im + \binom{s - i}{2}} {{\ell -i}\brack{\ell -s}}_q {{\ell -r}\brack{i}}_q=-1.$$
This identity may readily be shown using for example \cite[Thm 3.3]{andrews} after interchanging the summation order. In any case, it is clear that Equations \eqref{eq:delsarte} and \eqref{eq:weights} may not always give the easiest possible expression for the weights.
\end{example}

While for $t=1$ and $t=\ell$ all weights $\hat{w}_r(t;\ell,m)$ are easy to compare with one another, the same cannot be said in case $1 < t < \ell$. We formulate the following conjecture.
\begin{conjecture}\label{conj:weights}
Let $\ell \le m$ be positive integers and $t$ an integer satisfying $1 < t < \ell$. The the following hold:
\begin{enumerate}
\item All weights $\hat{w}_1(t;\ell,m),\dots,\hat{w}_\ell(t;\ell,m)$ are mutually distinct.
\item We have $\hat{w}_1(t;\ell,m) < \hat{w}_2(t;\ell,m) < \cdots < \hat{w}_{\ell-t+1}(t;\ell,m)$.
\item For all $\ell-t+2 \le r \le \ell$, the weight $\hat{w}_r(t;\ell,m)$ lies between $\hat{w}_{r-2}(t;\ell,m)$ and $\hat{w}_{r-1}(t;\ell,m)$.
\end{enumerate}
\end{conjecture}

\section{Minimum distance of determinantal codes}

Recall that in general for a linear code $C$ of length $n$, i.e., for a linear subspace $C$ of $\Fq^n$,
the \emph{Hamming weight} of a codeword $c=(c_1, \dots , c_n)$, denoted $\wH(c)$ is defined by
$$
\wH(c) : =  |\{i: c_i\ne 0\}|.
$$
The minimum distance of $C$, denoted $d(C)$, is defined by
\begin{eqnarray*}
d(C) &:= & \min\{\wH(c) : c\in C, \; c\ne 0\}.
\end{eqnarray*}
A consequence of Conjecture \ref{conj:weights} would also be that $\hat{w}_1(t;\ell,m)$ is the minimum distance of $\Cpt$. We will now show that this is indeed the case. We start by giving a rather compact expression for $\hat{w}_1(t;\ell,m)$.

\begin{proposition}\label{prop:w1hat}
Let $t,\ell$, and $m$ be integers satisfying $1 \le t \le \ell \le m$. Then
$$\hat{w}_1(t;\ell,m)=q^{\ell+m-2}\nu_{t-1}(\ell-1,m-1).$$%|\mathcal{D}_{t-1}(\ell-1,m-1)|.$$
\end{proposition}
\begin{proof}
First suppose that $t=1$. In this case Example \ref{ex:t=1} implies that $\hat{w}_1(1;\ell,m)=q^{m+\ell-2}$. On the other hand, using Equations \eqref{mur} and \eqref{mur2}, we see that $|\mathcal{D}_{t-1}(\ell-1,m-1)|=\mu_{0}(\ell-1,m-1)=1$, so the proposition holds for $t=1$.

From now on we assume that $t>1$ (implying that also $\ell >1$). We will show that
\begin{equation}\label{eq:wfrak1}
\hat{\mathfrak{w}}_1(t;\ell,m)=q^{\ell+m-2}\mu_{t-1}(\ell-1,m-1).
\end{equation}
Once we have shown this, the proposition follows using Equations \eqref{mur2} and \eqref{eq:weights}. Let $M = (m_{ij}) \in \Dt(\ell,m)$ and suppose that $\tau_1(M)=m_{11} \neq 0$. In that case, we may find uniquely determined square matrices
$$ A = \left(
\begin{array}{cccc}
1 & & &  \bf 0\\
a_1 & 1 & &  \\
\vdots & & \ddots & \\
a_{\ell-1} & \bf 0 & & 1
\end{array}
\right)
\ \makebox{and} \
B = \left(
\begin{array}{cccc}
1 & b_1 & \cdots  & b_{m-1} \\
   & 1 & &  \bf 0\\
 & & \ddots & \\
\bf 0 &   & & 1
\end{array}
\right),$$
such that
\begin{equation}\label{eq:mtilde}
AMB=\left(
\begin{array}{cccc}
m_{11} & 0 & \cdots  & 0 \\
0 &  & &  \\
0 & & \widetilde{M} & \\
0 &   & &
\end{array}
\right).
\end{equation}
The matrices $A$ and $B$ are indeed uniquely determined, since for $2 \le i \le \ell$ and $2 \le j \le m$ we have
$$0=(AMB)_{i1}=a_{i-1}(MB)_{11}+(MB)_{i1}=a_{i-1}m_{11}+m_{i1}$$
and
$$0=(AMB)_{1j}=(AM)_{11}b_{j-1}+(AM)_{1j}=m_{11}b_{j-1}+m_{1j}.$$
These equations determine the values of $a_1,\dots,b_{m}$ given the matrix $M$. The association of $\phi(M)=\widetilde{M}$ therefore is a well-defined map $$\phi: \{M \in \mathfrak{D}_t(\ell,m) \,|\, m_{11}\neq 0\} \rightarrow \mathfrak{D}_{t-1}(\ell-1,m-1).$$ The map $\phi$ is clearly surjective (one can for example choose $M$ as in the right hand side of Equation \eqref{eq:mtilde}), while the preimage of any matrix $\widetilde{M} \in \mathfrak{D}_{t-1}(\ell-1,m-1)$ consist of the $(q-1)q^{\ell+m-2}$ matrices of the form $A^{-1}MB^{-1}$, with $A$ and $B$ as above and again $M$ chosen as in the right-hand-side of Equation \eqref{eq:mtilde}. Equation \eqref{eq:wfrak1} (and hence the proposition) then follows, since
\begin{eqnarray*}
\mathfrak{w}_1(t;\ell,m) & = & |\{M \in \mathfrak{D}_t(\ell,m) \,|\, m_{11}\neq 0\}| \\
 & = & \sum_{\widetilde{M} \in \mathfrak{D}_{t-1}}|\phi^{-1}(\widetilde{M})|=|\mathfrak{D}_{t-1}(\ell-1,m-1)|(q-1)q^{\ell+m-2}.
\end{eqnarray*}
Equation \eqref{eq:wfrak1}, and hence the proposition, follows directly from this.
\end{proof}
Note that the expression for $\mathfrak{w}_1(t;\ell,m)$ from Equation \eqref{eq:delsarte} is considerably more involved that the expression obtained in the proof of Proposition \ref{prop:w1hat}.
%Comparing Equations \eqref{eq:delsarte} and \eqref{eq:wfrak1} yields an equation involving Gaussian binomials.
%$$q^{\ell+m-1}\mu_{t-1}(\ell-1,m-1)=\left( \mu_t (\ell,m) - \sum_{i=0}^{\ell} (-1)^{t-i} q^{im + \binom{t - i}{2}} {{\ell -i}\brack{\ell -t}}_q {{\ell -1}\brack{i}}_q  \right).$$
We now turn our attention to proving that $\hat{w}_1(t;\ell,m)$ actually is the minimum distance of the code $\Cpt$. The proof involves several identities concerning $\mathfrak{\hat{w}}(t;\ell,m)$ and $\hat{w}(t;\ell,m)$. The key is the following theorem in which the following quantity occurs:
$$A(r,t):=q^t \mathfrak{\hat{w}}_{r-1}(t;\ell-1,m-1)+q^{t-1}\left(\,\mu_t(\ell-1, m)-\mu_t(\ell-1,m-1)\,\right), \ \makebox{for $0 \le t < \ell$}.$$

\begin{theorem}\label{thm:keyrec}
Let $1 \le r \le \ell \le m$ and $1 \le t < \ell$, then
$$
\mathfrak{\hat{w}}_r(t;\ell,m)  =  A(r,t)-A(r,t-1)+ q^{m-1}\mu_{t-1}(\ell-1,m).
$$
\end{theorem}
\begin{proof}
Given a matrix $M=(m_{ij})\in \mathfrak D_t(\ell,m)$, we denote by $\psi(M)$ the matrix obtained from $M$ by deleting its $r$-th row. Since either $\psi(M) \in \mathfrak D_t(\ell-1,m)$ or $\psi(M) \in \mathfrak D_{t-1}(\ell-1,m)$, this defines a map $\psi:\mathfrak D_t(\ell,m) \to \mathfrak D_t(\ell-1,m)\bigsqcup \mathfrak D_{t-1}(\ell-1,m)$. It is not hard to see that $\psi$ is surjective. In fact:
\begin{equation}\label{eq:cardinality preimage}
 |\psi^{-1}(N)|=\left\{
\begin{array}{ll}
q^t           & \makebox{if $N\in \mathfrak D_t(\ell-1,m)$}\\
q^{m}-q^{t-1} & \makebox{if $N\in \mathfrak D_{t-1}(\ell-1,m)$}
\end{array}
\right.,
\end{equation}
since if $N\in \mathfrak D_t(\ell-1,m)$ we obtain all elements of $\psi^{-1}(N)$ by adding a row from the rowspace of $N$, while if $N\in \mathfrak D_{t-1}(\ell-1,m)$ we obtain all elements of $\psi^{-1}(N)$ by adding any row not from the rowspace of $N$.

We will now prove the theorem by carefully counting the number of matrices $M \in \mathfrak D_t(\ell,m)$ such that $\tau_r(M) \neq 0$, thus computing $\mathfrak w_r(t;\ell,m)$. The theorem then follows easily, since $\mathfrak w_r(t;\ell,m)=(q-1)\mathfrak{\hat{w}}_r(t;\ell,m)$. We distinguish four cases:

\begin{enumerate}
\item[Case 1:] The $r$-th column of $\psi(M)$ is zero and $\psi(M)$ has rank $t$,
\item[Case 2:] The $r$-th column of $\psi(M)$ is zero and $\psi(M)$ has rank $t-1$,
\item[Case 3:] The $r$-th column of $\psi(M)$ is non zero and $\psi(M)$ has rank $t$,
\item[Case 4:] The $r$-th column of $\psi(M)$ is non zero and $\psi(M)$ has rank $t-1$.
\end{enumerate}

Case 1: The $r$-th column of $\psi(M)$ is zero and $\psi(M)$ has rank $t$. In this case $m_{rr}=0$, since otherwise $\rk(\psi(M))\neq \rk(M)$. Therefore $\tau_r(M) \neq 0$ if and only if $\tau_{r-1}(\psi(M)) \neq 0$. By Equation \eqref{eq:cardinality preimage}, we find the following contribution to $\mathfrak w_r(t;\ell,m)$:
\begin{equation}\label{eq:case1}
q^t \mathfrak w_{r-1}(t;\ell-1,m-1).
\end{equation}

Case 2: The $r$-th column of $\psi(M)$ is zero and $\psi(M)$ has rank $t-1$. If $m_{rr}=0$, then by a similar reasoning as in case 1, we find a contribution to $\mathfrak w_r(t;\ell,m)$ of magnitude
\begin{equation}\label{eq:case2a}
(q^{m-1}-q^{t-1})\mathfrak w_{r-1}(t-1;\ell-1,m-1).
\end{equation}
If $m_{rr}\neq 0$, the situation is more complicated. If namely $\tau_{r-1}(\psi(M))=0$, then $\tau_r(M)\neq 0$ if and only if $m_{rr}\neq 0$. Since $\rk(\psi(M))=t-1$ and the $r$-th column of $\psi(M)$ is zero, all $q^{m-1}(q-1)$ matrices with nonzero $(r,r)$-th entry are in $\psi^{-1}(\psi(M))$. This gives a contribution to $\mathfrak w_r(t;\ell,m)$ of magnitude
\begin{equation}\label{eq:case2b}
q^{m-1}(q-1)\left(\mu_{t-1}(\ell-1,m-1)- \mathfrak w_{r-1}(t-1;\ell-1,m-1)\right).
\end{equation}
If on the other hand $\tau_{r-1}(\psi(M))\neq 0$, then $\tau_r(M) \neq 0$ if and only if $m_{rr} \neq \tau_{r-1}(\psi(M))$. Since we already assumed that $m_{rr} \neq 0$, we find a contribution to $\mathfrak w_r(t;\ell,m)$ of magnitude
\begin{equation}\label{eq:case2c}
q^{m-1}(q-2)\mathfrak w_{r-1}(t-1;\ell-1,m-1).
\end{equation}

Case 3: The $r$-th column of $\psi(M)$ is non zero and $\psi(M)$ has rank $t$. Since the $r$-th column of $\phi(M)$ is non zero, the $r$-th coordinates of elements from the row space of $\psi(M)$ are distributed evenly over the elements of $\Fq$. This implies that regardless of the value of $\tau_{r-1}(\psi(M))$, a $(q-1)/q$-th fraction of the matrices in $\psi^{-1}(\psi(M))$ contribute to $\mathfrak w_r(t;\ell,m)$. In total we find the contribution:
\begin{equation}\label{eq:case3}
q^{t-1}(q-1)\left(\mu_t(\ell-1,m)-\mu_t(\ell-1,m-1)\right).
\end{equation}

Case 4: The $r$-th column of $\psi(M)$ is non zero and $\psi(M)$ has rank $t-1$. Just as in case 3, since the $r$-th column of $\psi(M)$ is non zero, the $r$-th coordinates of elements from the row space of $\psi(M)$ are distributed evenly over the elements of $\Fq$. Therefore also the $r$-th coordinates of elements not from the row space of $\psi(M)$ are distributed evenly over the elements of $\Fq$. By a similar reasoning as in case 3, we find a contribution to $\mathfrak w_r(t;\ell,m)$ of magnitude:
\begin{equation}\label{eq:case4}
(q^{m-1}-q^{t-2})(q-1)\left(\mu_{t-1}(\ell-1,m)-\mu_{t-1}(\ell-1,m-1)\right).
\end{equation}
Adding all contributions to $\mathfrak{w}_r(t;\ell,m)$ from Equations \eqref{eq:case1},\eqref{eq:case2a},\eqref{eq:case2b},\eqref{eq:case2c},\eqref{eq:case3}, and \eqref{eq:case4}, the theorem follows.
\end{proof}

\begin{corollary}
Let $1 \le r \le \ell \le m$ and $1 \le t < \ell$. Then
$$\hat{w}_r(t;\ell,m)=A(r,t)+q^{m-1}\nu_{t-1}(\ell-1,m-1).$$%|\mathcal{D}_{t-1}(\ell-1,m)|.$$
\end{corollary}
\begin{proof}
By Equation \eqref{eq:weights} and Theorem \ref{thm:keyrec} we have
\begin{eqnarray*}
\hat{w}_r(t;\ell,m) & = & \sum_{s=1}^t \hat{\mathfrak w}_r(s;\ell,m)\\
 & = & \sum_{s=1}^t \left( A(r,s)-A(r,s-1)+ q^{m-1}\mu_{s-1}(\ell-1,m) \right)\\
 & = & A(r,t)-A(r,0)+q^{m-1}\sum_{s=1}^t\mu_{s-1}(\ell-1,m).
\end{eqnarray*}
The corollary now follows by the definition of $A(r,t)$ and Equation \eqref{mur2}.
\end{proof}

\begin{corollary}\label{cor:difference}
Let $1 \le s \le r \le \ell$ and $1 \le t < \ell$, then
$$\hat{w}_r(t;\ell,m)-\hat{w}_s(t;\ell,m)=q^t\left(\mathfrak{\hat{w}}_{r-1}(t;\ell-1,m-1)-\mathfrak{\hat{w}}_{s-1}(t;\ell-1,m-1)\right).$$
In particular
$$\hat{w}_r(t;\ell,m)-\hat{w}_1(t;\ell,m)=q^t\mathfrak{\hat{w}}_{r-1}(t;\ell-1,m-1).$$
\end{corollary}
\begin{proof}
Using the previous corollary, we see that
\begin{eqnarray*}
\hat{w}_r(t;\ell,m)-\hat{w}_{s}(t;\ell,m) & = & A(r,t)-A(s,t)\\
 & = & q^t\left(\mathfrak{\hat{w}}_{r-1}(t;\ell-1,m-1)-\mathfrak{\hat{w}}_{s-1}(t;\ell-1,m-1)\right).
\end{eqnarray*}
This yields the first part of the corollary. % now follows. 
The second part follows directly by choosing $s=1$.
\end{proof}

We are now ready to prove our main theorem on the minimum distance.

\begin{theorem}\label{thm:mindist}
Let $1 \le r \le \ell \le m$ and $1 \le t \le \ell$. Then the minimum distance $\hat{d}$ of the code $\Cpt$ is given by $$\hat{d}=q^{\ell+m-2}\nu_{t-1}(\ell-1,m-1).$$%|\mathcal{D}_{t-1}(\ell-1,m-1)|.$$
\end{theorem}
\begin{proof}
We already know that the only $\ell$ nonzero weights occurring in code $\Cpt$ are $\hat{w}_1(t;\ell,m), \dots, \hat{w}_\ell(t;\ell,m)$. Moreover, in case $t=\ell$, we already know from Example \ref{ex:t=ell} that the minimum distance is given by
$$\hat{w}_1(\ell;\ell,m)=q^{\ell m-1}=q^{\ell+m-2}\nu_{t-1}(\ell-1,m-1).$$%|\mathcal{D}_{\ell-1}(\ell-1,m-1)|.$$
Therefore we may assume $t<\ell$. However, in this case the second part of Corollary \ref{cor:difference} implies that $\hat{w}_1(t;\ell,m)$ cannot be larger than any of the other weights, since
$$\hat{w}_r(t;\ell,m)-\hat{w}_1(t;\ell,m)=q^t\mathfrak{\hat{w}}_{r-1}(t;\ell-1,m-1) \ge 0.$$
The theorem then follows from Proposition \ref{prop:w1hat}.
\end{proof}

The above theorem gives a start to proving Conjecture \ref{conj:weights}. Exploring the above methods, we can do a little more as well as gain some information about codewords of minimum weight in $\Cpt$.
\begin{proposition}
Let $1 \le t < \ell$, then $\hat{w}_1(t;\ell,m)<\hat{w}_r(t;\ell,m).$ The code $\Cpt$ has exactly $\mu_1(\ell,m)$ codewords of minimum weight and these codewords generate the entire code. More precisely, any codeword in $\Cpt$ is the sum of at most $\ell$ minimum weight codewords.
\end{proposition}
\begin{proof}
Choosing $s=1$ in Corollary \ref{cor:difference} and $r \ge 2$, we obtain that $$\hat{w}_r(t;\ell,m)-\hat{w}_1(t;\ell,m)=q^t\mathfrak{\hat{w}}_{r-1}(t;\ell-1,m-1),$$ so the first part of the proposition follows once we have shown that $\mathfrak{\hat{w}}_{r-1}(t;\ell-1,m-1)>0$. In order to this, it is sufficient to produce one $\ell-1 \times m-1$ matrix $M$ of rank $t$ such that $\tau_{r-1}(M) \neq 0$. However, this is easy to do: Let $P=(p_{ij})$ be a $t \times t$ permutation matrix corresponding to a permutation on $t$ elements that fixes $1$, but does not have other fixed points. Then $p_{11}=1$, while any other diagonal element is zero. Now take $M=(m_{ij})$ to be the $\ell-1 \times m-1$ matrix such that $m_{ij}=p_{ij}$ if $i < \ell$ and $j < m$, while $m_{ij}=0$ otherwise. Then for any $r \ge 2$, we have $\tau_{r-1}(M)=1$, which is exactly what we wanted to show.

Now that we know that $\hat{w}_1(t;\ell,m)$ is strictly smaller than all other nonzero weights, the minimum weight codewords are exactly those $\hat{c}_f$ such that $f$ has a coefficient matrix of rank $1$. This gives exactly $\mu_1(\ell,m)$ possibilities for $f$ and hence for $\hat{c}_f$. Now let $\hat{c}_f \in \Cpt$ be given. Assume that $f$ has coefficient matrix $F=(f_{ij})$ of rank $r$. Since any matrix of rank $r$ can be written as the sum of $r$ matrices of rank $1$, we can write $f=g_1+\cdots+g_r$ for certain $g_1,\dots,g_r \in \Fq[X]_1$ all having a coefficient matrix of rank $1$. This implies that $\hat{c}_f=\hat{c}_{g_1}+\cdots+\hat{c}_{g_r}$, implying the second part of the proposition.
\end{proof}

The case $t=\ell$ is not covered by the above proposition. However, in that case it follows directly from Example \ref{ex:t=ell} that $\hat{w}_1(t;\ell,m)=\hat{w}_r(t;\ell,m)$ for any $r\ge 2$. The number of codewords of minimum weight is therefore given by $q^{\ell m-1}-1$ and they clearly generate the code.

\begin{remark}
\label{rem:Delsarte}
{\rm 
If Conjecture \ref{conj:weights} is true, then Corollary \ref{cor:difference} implies that the quantities $\mathfrak{w}_r(t;\ell,m)$ would have a similar behaviour. More precisely, let $1 \le t \le \ell$, then it would hold that:
\begin{enumerate}
\item[(i)] All weights $\mathfrak{w}_1(t;\ell,m),\dots,\mathfrak{w}_\ell(t;\ell,m)$ are mutually distinct.
\item[(ii)] $\mathfrak{w}_1(t;\ell,m) < \mathfrak{w}_2(t;\ell,m) < \cdots < \mathfrak{w}_{\ell-t+1}(t;\ell,m)$.
\item[(iii)] For $\ell-t+2 \le r \le \ell$, the weight $\mathfrak{w}_r(t;\ell,m)$ lies between $\mathfrak{w}_{r-2}(t;\ell,m)$ and $\mathfrak{w}_{r-1}(t;\ell,m)$.
\end{enumerate}
%This could be interesting because of a relation with the theory of association schemes.
We remark that these assertions have a bearing on the eigenvalues of the association scheme of bilinear forms (using the rank metric as distance) \cite[Section 9.5.A]{BCN}.
Indeed, the eigenvalues of this association scheme are precisely given by the expressions
\begin{equation}\label{eq:ptr_delsarte}
P_t(r):=\sum_{i=0}^{\ell} (-1)^{t-i} q^{im + \binom{t - i}{2}} {{\ell -i}\brack{\ell -t}}_q {{\ell -r}\brack{i}}_q
\end{equation}
occurring in Equation \eqref{eq:delsarte}.
%From the above these properties would become clear for the association scheme of bilinear forms, thus leading to a better understanding.
For a general association scheme, it is not known how its eigenvalues are ordered or if they are all distinct. See \cite{brouwerfiol} for a study of the nondistinctness of some of such eigenvalues.
It is known in general that the eigenvalues exhibit sign changes (see for example \cite[Prop. 11.6.2]{brouwerhamers}), which is in consonance with the conjectured behaviour of the $\mathfrak{w}_r(t;\ell,m)$ in part (iii) above. 
}
\end{remark}

\section{Generalized Hamming weights of determinantal codes}

We now turn our attention to the computation of several of the generalized Hamming weights of the determinantal code $\Cpt$. Given that it was not trivial to compute the minimum distance, this may seem ambitious, but it turns out that we can use the work carried out in the previous section and compute the first $m$ generalized Hamming weights without much extra effort.

For a linear code $C$ of length $n$ and dimension $k$ the \emph{support weight} of any $D\subseteq C$, denoted $\Vert D\Vert$, is defined by
$$
\Vert D \Vert  :=  |\{ i: \text{there exists } c\in D \text{ with } c_i \ne0\}|.
$$
For $1 \le s \le k$ the $s^{\rm th}$ \emph{generalized Hamming weight} of $C$, denoted $d_s(C)$, is defined by
\begin{eqnarray*}
d_s(C)&:=& \min\{\Vert D \Vert : D \text{ is a subcode of }C \text{ with } \dim D=s\}.
\end{eqnarray*}
We have $d_1(C)=d(C)$, the minimum distance of the code $C$, while $d_k(C)=n$ if the code $C$ is nondegenerate.
\begin{theorem}
\label{HigherWts}
For $s =1, \dots , m$, the $s$-th generalized Hamming weight $\hat{d}_s$ of $\Cpt$ is given by
$$
\hat{d}_s = q^{\ell + m -s -1}\nu_{t-1}(\ell-1,m-1).%|\mathcal{D}_{t-1}(\ell-1,m-1)|.
$$
\end{theorem}
\begin{proof} 
Fix $s\in \{1, \dots , m\}$ and let $L_s$ be the $s$-dimensional subspace of $\Fq[X]_1$ generated by $X_{11}, \dots , X_{1s}$. Also let $D_s = \Ev (L_s)$ be the corresponding subcode of $\Cpt$. Since $\Ev$ is injective and linear, $\dim D_s = s$. Moreover, since the coefficient matrix of
any $f\in L_s$ different from zero has rank one, it follows from Fact \ref{fact:fromdetcodes1} that $\wH(\hat{c}_f) = \hat{w}_1(t;\ell,m)$. Using the formula for the support weight of an $s$-dimensional subcode given in for example \cite[Lemma 12]{GPP}, we obtain
$$
\| D_s \| = \frac{1}{q^s - q^{s-1}} \sum_{c\in D_s} \wH(c) = \frac{q^s-1}{q^s - q^{s-1}}\hat{w}_1(t;\ell,m).
$$
On the other hand, since $\hat{w}_1(t;\ell,m)$ is the minimum distance of $\Cpt$, it holds for any subspace $D \subset \Cpt$ of dimension $s$ that
$$
\| D \| = \frac{1}{q^s - q^{s-1}} \sum_{c\in D_s} \wH(c) \ge \frac{q^s-1}{q^s - q^{s-1}} \hat{w}_1(t;\ell,m).
$$
Using Theorem \ref{thm:mindist}, we obtain the stated formula.
\end{proof}

Though more involved, it is possible to obtain the $m+1$-th generalized Hamming weight as well:

\begin{proposition}
Suppose that $\ell \ge 2$, then the $(m+1)$-th generalized Hamming weight $\hat{d}_{m+1}$ of $\Cpt$ is given by
$$
\hat{d}_{m+1} = \hat{d}_m+q^{\ell-2}\nu_{t-1}(\ell-1,m-1)%|\mathcal{D}_{t-1}(\ell-1,m-1)|
+(q^{m-1}-1)q^{\ell+t-1}\mu_{t-1}(\ell-1,m-1).
$$
\end{proposition}
\begin{proof}
Let $L_{m+1} \subset \Fq[X]_1$ be the $m+1$-dimensional space generated by $X_{11},\dots,X_{1m}$, $X_{21}$ and write $D_{m+1}=\Ev(L_{m+1})$. As in the proof of \cite[Lem. 2]{detcodes1} ome readily sees that $L_{m+1}$ contains $1$ function with coefficient matrix of rank $0$ (namely the zero function) and exactly $q^m+q^2-q-1$ (resp. $(q-1)(q^m-q)=q^{m+1}-q^m+q^2+q$) functions with coefficient matrix of rank $1$ (resp. rank $2$). Therefore we obtain that
\begin{eqnarray*}
\hat{d}_{m+1} & \le & \frac{1}{q^{m+1} - q^{m}} \sum_{c\in D_{m+1}} \wH(c)\\
& = & \frac{1}{q^{m+1} - q^{m}}\left( (q^m+q^2-q-1)\hat{w}_1(t;\ell,m)+(q-1)(q^m-q)\hat{w}_2(t;\ell,m)\right)\\
& = & \hat{d}_m+\frac{\hat{w}_1(t;\ell,m)}{q^{m}}+\frac{q^{m-1}-1}{q^{m-1}}(\hat{w}_2(t;\ell,m)-\hat{w}_1(t;\ell,m))\\
& = & \hat{d}_m+q^{\ell-2}\nu_{t-1}(\ell-1,m-1)%|\mathcal{D}_{t-1}(\ell-1,m-1)|
+(q^{m-1}-1)q^{\ell+t-1}\mu_{t-1}(\ell-1,m-1).
\end{eqnarray*}
Where in the last equality we used Proposition \ref{prop:w1hat}, Corollary \ref{cor:difference} and Equation \eqref{eq:wfrak1}.
On the other hand, in \cite[Lem. 4]{detcodes1} it is stated that any $m+1$-dimensional subspace of $\Mat_{\ell \times m}$ contains at most $q^{m}+q^2 - q - 1$ matrices of rank $1$ and
at least  $\left(q^{m} - q \right)(q-1)$ matrices of rank $\ge 2$. This implies the desired result.
\end{proof}

Finally, we will determine the final $tm$ generalized Hamming weights. While before, we have mainly used the description of $\Cpt$ as evaluation code, it turns out to be more convenient now to use the geometric description of $\Cpt$ as projective system coming from $\Dtp$. The approach is similar to the one given Appendix A in \cite{DG}, though there a completely different class of codes was considered. The following lemma holds the key:

\begin{lemma}\label{lem:projcontain}
The projective variety $\Dtp(\ell,m) \subset \PP^{\ell m-1}$ contains the projective space $\PP^{tm-1}$.
\end{lemma}
\begin{proof}
Since any matrix $M \in \Mat_{\ell \times m}$ with at most $t$ nonzero rows is in $\Dt(\ell,m)$, we see that
$$\{ (m_{ij}) \, |\, m_{ij}=0 \ \makebox{for} \ 1 \le i \le \ell-t, \ 1 \le j \le m \} \subset \Dt(\ell,m).$$
Passing to homogeneous coordinates, the lemma follows.
\end{proof}

In the language of projective systems, the $s$-th Generalized Hamming weight can be described rather elegantly. If $C$ is a code of length $n$ and dimension $k$ described by a projective system $X \subset \PP^{k-1}$, then
\begin{equation}\label{eq:genHam}
d_s(C)=n-\max_{{\rm codim}L=s} |X \cap L|,
\end{equation}
where the maximum is taken over all planes $L \subset \PP^{k-1}$ of codimension $s$ (see \cite{TV1,TV2} for more details). This description, combined with the previous lemma, gives the following result.

\begin{theorem}\label{thm:lastgenHam}
Let $1 \le t \le \ell \le m$ be given integers and $(\ell-t)m \le s \le \ell m$, then $\hat{d}_s$, the $s$-th generalized Hamming weight of $\Cpt$, is given by
$$\hat{d}_s=\hat{n}-\sum_{i=0}^{\ell m-s-1}q^i.$$
\end{theorem}
\begin{proof}
First of all note that if $s=\ell m$, we have $\hat{d}_s=\hat{n}$, since the code $\Cpt$ is nondegenerate (see Fact \ref{fact:fromdetcodes1}). Therefore, we assume that $(\ell-t)m \le s< \ell m$. If $(\ell-t)m \le s \le \ell m$, there exists a subplane $L=L_s$ of codimension $\ell m-s-1$ contained in $\Dtp(\ell,m)$ by Lemma \ref{lem:projcontain}. Clearly this choice of $L$ in Equation \eqref{eq:genHam} leads directly to the $s$-th generalized Hamming weight, since in this case $\Dtp(\ell,m) \cap L_s=L_s$. Since $|L_s|=|\PP^{\ell m -1 -s}|=\sum_{i=0}^{\ell m-s-1}q^i$, the expression for $\hat{d}_s$ follows.
\end{proof}

\begin{corollary}
The minimum distance of $\Cpt^{\perp}$ equals $3$.
\end{corollary}
\begin{proof}
From Theorem \ref{thm:lastgenHam}, we see that $d_{\ell m-2}=\hat{n}-q-1$, $d_{\ell m}=\hat{n}-1$, and $d_{\ell m}=\hat{n}$. By duality this implies that the first generalized Hamming weights of $\Cpt^{perp}$ (that is to say its minimum distance) is given by $\hat{d}_1^{\perp}=3$.
\end{proof}

\begin{corollary}
In case $t=\ell-1$ all generalized Hamming weights of $\Cpt$ are known and given by
$$d_s= \left\{
\begin{array}{ll}
q^{\ell + m -s -1}\nu_{\ell-2}(\ell-1,m-1)%|\mathcal{D}_{\ell-2}(\ell-1,m-1)|
, & \makebox{if $1 \le s \le m$},\\

\\

\hat{n}-\sum_{i=0}^{\ell m-s-1}q^i, & \makebox{otherwise.}
\end{array}
\right.$$
\end{corollary}
\begin{proof}
This follows by combining Theorems \ref{HigherWts} and \ref{thm:lastgenHam}.
\end{proof}

Also in case $t=\ell$ Theorem \ref{thm:lastgenHam} gives all generalized Hamming weights of $\Cpt$. However, in this case $\Cpt$ is simply a first order projective Reed--Muller code for which all generalized Hamming weights are well known.

\section*{Acknowledgments}
The first and the second named authors %We
are grateful to the Indian Institute of Technology Bombay and the % IIT Bombay and DTU Lyngby
Technical University of Denmark, respectively, for the warm hospitality and support of short visits to these institutions where some of this work was done. %The authors
We %would also like to
thank Andries Brouwer for helpful correspondence and bringing references \cite{brouwerfiol} and \cite{brouwerhamers} to our attention.

\section*{Appendix}

In this appendix we give a self-contained computation of the quantity $\mathfrak{w}_r(t;\ell,m)$. The %employed
method we use is different from the one Delsarte used in \cite{D} and consequently gives rise to an alternative formula to the one Delsarte obtained. Essentially our methods concerns the study of a refined description of the sets $\mathfrak{D}_t(\ell,m)$ as the union of disjoint subsets. For $M \in \Mat_{\ell \times m},$ and $1 \le r \le \ell$, we denote by $\underline{M}_r$ the $r \times m$ matrix obtained by taking the first $r$ rows of $M$. We use this to define the following quantities:

\begin{definition}
Let $1 \le t \le \ell \le m$, $1 \le r \le \ell$ and $1 \le s \le t$. Then we define
$$\mathfrak{D}_{t}(\ell,m;r,s)=\{M \in \mathfrak{D}_{t}(\ell,m) \, | \, \rk(\underline{M}_r)=s\}.$$
Further we define
$$\mathfrak{w}^{(s)}_r(t;\ell,m)=w_H( (\tau_r(M))_{M \in \mathfrak{D}_{t}(\ell,m;r,s)} ),$$
with as before $\tau_r=X_{11}+\cdots+X_{rr}.$
\end{definition}
Note that
\begin{equation}\label{eq:sum ws is w}
\mathfrak{w}_r(t;\ell,m)=\sum_{s=1}^r\mathfrak{w}_r^{(s)}(t;\ell,m).
\end{equation}

\begin{proposition}\label{prop:msrtlm}
Let $r,s,t,\ell$, and $m$ be integers satisfying $1 \le t \le \ell \le m$, $1 \le  r \le \ell$, and $1 \le s \le t$. Then we have
$$|\mathfrak{D}_{t}(\ell,m;r,s)|=\dfrac{[m]_q!}{[m-t]_q!}q^{s(\ell-r)}q^{\binom{s}{2}}q^{\binom{t-s}{2}}
{{r}\brack{s}}_q {{\ell-r}\brack{t-s}}_q.$$
\end{proposition}
\begin{proof}
We choose $r$ arbitrarily and treat it as a fixed constant from now on. If $\ell<r$, then $|\mathfrak{D}_{t}(\ell,m;r,s)|=0$, which fits with the formula. Therefore we suppose from now on that $\ell \ge r$ and we will prove the proposition with induction on $\ell$ for values $\ell \ge r$.

\bigskip

Induction basis: If $\ell=r$, then $\mathfrak{D}_{t}(\ell,m;r,s)=\mathfrak{D}_{t}(\ell,m)$ if $s=t$, while otherwise $\mathfrak{D}_{t}(\ell,m;r,s)=\emptyset$. In the latter case the proposed formula gives the correct value $0$, while if $s=t$ also the correct value from Equation \eqref{mur} is recovered. This completes the induction basis.

\bigskip

Induction step: Suppose $\ell>r$. Let $A \in \mathfrak{D}_{t}(\ell,m;r,s)$. Then $\underline{A}_{\ell-1}$ is an element of $\mathfrak{D}_{t}(\ell-1,m;r,s)$ or of $\mathfrak{D}_{t-1}(\ell-1,m;r,s)$. Conversely, a matrix from $\mathfrak{D}_{t}(\ell-1,m;r,s)$ can be extended (by adding a row from the rowspace of the matrix) to an element of $\mathfrak{D}_{t}(\ell,m;r,s)$ in exactly $q^t$ ways, while a matrix from $\mathfrak{D}_{t-1}(\ell-1,m;r,s)$ can be extended (by adding a row not from the rowspace of the matrix) to an element of $\mathfrak{D}_{t}(\ell,m;r,s)$ in exactly $q^m-q^{t-1}$ ways. Therefore
$$|\mathfrak{D}_{t}(\ell,m;r,s)|=q^t|\mathfrak{D}_{t}(\ell-1,m;r,s)|+(q^m-q^{t-1})|\mathfrak{D}_{t-1}(\ell-1,m;r,s)|.$$
Using the induction hypothesis, this equation implies:
\begin{eqnarray*}
|\mathfrak{D}_{t}(\ell,m;r,s)| & = &\dfrac{[m]_q!}{[m-t]_q!}q^{s(\ell-r)}q^{\binom{s}{2}}q^{\binom{t-s}{2}}
{{r}\brack{s}}_q {{\ell-r}\brack{t-s}}_q\cdot \\
 & & \left(q^t q^{-s} \dfrac{q^{\ell-r-t+s}-1}{q^{\ell-r}-1}+(q^m-q^{t-1})\dfrac{1}{q^{m-t+1}}q^{-s}q^{-(t-1-s)}\dfrac{q^{t-s}-1}{q^{\ell-r}-1}\right).
\end{eqnarray*}
However, the term between the brackets is easily seen to be equal to $1$, concluding the inductive proof.
\end{proof}

The key argument in the induction step above can also be used to prove the following.

\begin{lemma}\label{lem:rec}
Let $r,s,t,\ell$, and $m$ be integers satisfying $1 \le t \le \ell \le m$, $1 \le  r \le \ell$, and $1 \le s \le t$. Then we have
$$\mathfrak w_r^{(s)}(t;\ell,m)=q^t\mathfrak w^{(s)}_r(t;\ell-1,m)+(q^m-q^{t-1})\mathfrak w^{(s)}_r(t-1;\ell-1,m), \, \makebox{if $\ell>r$}$$
and
$$\mathfrak w_r(t;\ell,m)=q^t\mathfrak w_r(t;\ell-1,m)+(q^m-q^{t-1})\mathfrak w_r(t-1;\ell-1,m), \, \makebox{if $\ell>r$}.$$
\end{lemma}
\begin{proof}
In the proof of Proposition \ref{prop:msrtlm}, we have seen that any matrix from $\mathfrak{D}_{t}(\ell-1,m;r,s)$ can be extended to an element of $\mathfrak{D}_{t}(\ell,m;r,s)$ in exactly $q^t$ ways, while a matrix from $\mathfrak{D}_{t-1}(\ell-1,m;r,s)$ can be extended to an element of $\mathfrak{D}_{t}(\ell,m;r,s)$ in $q^m-q^{t-1}$ ways. If $\ell > r$ the value of $\tau_r$ is the same for the original matrix and its extension. This immediately implies the first equation in the lemma. The second one follows from the first one using Equation \eqref{eq:sum ws is w}.
\end{proof}

\begin{remark}
By interchanging the roles of rows and columns, one can also show that
$$\mathfrak w_r^{(s)}(t;\ell,m)=q^t\mathfrak w^{(s)}_r(t;\ell,m-1)+(q^\ell-q^{t-1})\mathfrak w^{(s)}_r(t-1;\ell,m-1), \, \makebox{if $m>r$},$$
and
$$\mathfrak w_r(t;\ell,m)=q^t\mathfrak w_r(t;\ell-1,m)+(q^\ell-q^{t-1})\mathfrak w_r(t-1;\ell,m-1), \, \makebox{if $m>r$}.$$
\end{remark}

We will now derive a closed expression for the quantities $\mathfrak{w}_s(r,t;\ell,m)$. Like in the proof of Proposition \ref{prop:msrtlm}, we will use an inductive argument with base $r=\ell$. This explains why we first settle this case separately.

\begin{proposition}\label{prop:frak w l is r}
Let $s,t,\ell$, and $m$ be integers satisfying $1 \le \ell \le m$ and $1 \le s \le t$. Then we have
$$\mathfrak{w}_{\ell}^{(s)}(t;\ell,m)=0, \ \makebox{ if $t \neq s$,}$$
while
\begin{eqnarray*}
\mathfrak{w}_{\ell}^{(t)}(t;\ell,m) & = & \mathfrak{w}_{\ell}(t;\ell,m) = \dfrac{q-1}{q}\left(\mu_t(\ell,m) - (-1)^tq^{\binom{t}{2}}{{\ell}\brack{t}}_q \right).
\end{eqnarray*}
\end{proposition}
\begin{proof}
We have already seen that $\mathfrak{D}_{t}(\ell,m;r,s)=\mathfrak{D}_{t}(\ell,m)$ if $s=t$, while otherwise $\mathfrak{D}_{t}(\ell,m;r,s)=\emptyset$. Therefore the first part of the proposition follows, as well as the identity $\mathfrak{w}_{\ell}^{(t)}(t;\ell,m)=\mathfrak{w}_{\ell}(t;\ell,m)$. Now we prove that
$$
\mathfrak{w}_{\ell}(t;\ell,m) = \dfrac{q-1}{q}\left(\mu_t(\ell,m) - (-1)^tq^{\binom{t}{2}}{{\ell}\brack{t}}_q \right)
$$
with induction on $\ell$.

\bigskip
Induction basis: if $\ell=1$ (implying that $t=1$ as well), Proposition \ref{prop:w1hat} (or a direct computation) implies that $\mathfrak{w}_{1}(t;1,m)=(q-1)q^{m-1}$, which fits with the formula we wish to show.

\bigskip
Induction step: Assume that the formula holds for $\ell-1$. Using Theorem \ref{thm:keyrec} in the special case that $r=\ell$, we see that $$\mathfrak{w}_{\ell}(t;\ell,m)=q^t\mathfrak{w}_{\ell-1}(t;\ell-1,m-1)-q^{t-1}\mathfrak{w}_{\ell-1}(t-1;\ell-1,m-1)+A,$$
where $A$ is easily seen to be equal to
$$A=\frac{q-1}{q}\left(\mu_t(\ell,m)-q^t\mu_t(\ell-1,m-1)+q^{t-1}\mu_{t-1}(\ell-1,m-1)\right),$$
using the identity $\mu_t(\ell,m)=q^t \mu_t(\ell-1,m)+(q^m-q^{t-1})\mu_{t-1}(\ell-1,m).$
The induction hypothesis now implies that
\begin{eqnarray*}
\mathfrak{w}_{\ell}(t;\ell,m) & = & \frac{q-1}{q}\left( \mu_t(\ell,m) - q^t(-1)^tq^{\binom{t}{2}}{{\ell-1}\brack{t}}_q + q^{t-1}(-1)^{t-1}q^{\binom{t-1}{2}}{{\ell-1}\brack{t-1}}_q\right)\\
& = & \frac{q-1}{q}\left( \mu_t(\ell,m)- (-1)^tq^{\binom{t}{2}}\left( q^t{{\ell-1}\brack{t}}_q + {{\ell-1}\brack{t-1}}_q \right)\right)\\
& = & \frac{q-1}{q}\left( \mu_t(\ell,m)- (-1)^tq^{\binom{t}{2}}{{\ell}\brack{t}}_q \right),
\end{eqnarray*}
which is what we wanted to show.
\end{proof}

Now that the case $r=\ell$ is settled, we deal with the general case.

\begin{theorem}\label{thm:formula ws}
\begin{eqnarray*}
\mathfrak{w}_r^{(s)}(t;\ell,m) & = & \dfrac{q-1}{q}q^{\binom{s}{2}}\left(\dfrac{[m]_q!}{[m-t]_q!} - (-1)^s \dfrac{[m-s]_q!}{[m-t]_q!} \right)q^{s(\ell-r)}q^{\binom{t-s}{2}}
{{r}\brack{s}}_q {{\ell-r}\brack{t-s}}_q.
%\\
% & = & q^{m-1} q^{\binom{s-1}{2}} \left( \sum_{i=1}^s \dfrac{[m-i]!(-1)^{i+1}q^{s-i}}{[m-t]!} \right) q^{s(\ell-r)}q^{\binom{t-s}{2}}
%{{r}\brack{s}} {{\ell-r}\brack{t-s}}
\end{eqnarray*}
\end{theorem}
\begin{proof}
We prove the theorem by induction on $\ell$. If $\ell<r$, $\mathfrak{w}_r^{(s)}(t;\ell,m)=0$, which is consistent with the formula. If $\ell=r$, we have $\mathfrak{w}_r^{(s)}(t;\ell,m)=0$ if $s \neq t$ and $\mathfrak{w}_r^{(s)}(t;\ell,m)=\mathfrak{w}_r(t;\ell,m)$ if $s=t$. Using Proposition \ref{prop:frak w l is r} we see that the case $\ell=r$ of the theorem is valid.

\bigskip

Now suppose $\ell>r$. We may then apply Lemma \ref{lem:rec} and apply the induction hypothesis.
Performing very similar computations as in the proof of Proposition \ref{prop:msrtlm}, the induction step follows.
\end{proof}

We can now state our alternative formula for $\mathfrak{w}_r(t;\ell,m)$.

\begin{theorem}\label{thm:general formula}
We have
\begin{eqnarray*}
\mathfrak{w}_r(t;\ell,m) & = & \frac{q-1}{q}\sum_{s=1}^r q^{\binom{s}{2}}\left(\dfrac{[m]_q!}{[m-t]_q!} - (-1)^s \dfrac{[m-s]_q!}{[m-t]_q!} \right)q^{s(\ell-r)}q^{\binom{t-s}{2}}
{{r}\brack{s}}_q {{\ell-r}\brack{t-s}}_q\\
 & = &\dfrac{q-1}{q}\left(\mu_t(\ell,m)  -\sum_{s=0}^r q^{\binom{s}{2}}(-1)^s \dfrac{[m-s]_q!}{[m-t]_q!}q^{s(\ell-r)}q^{\binom{t-s}{2}}
{{r}\brack{s}}_q {{\ell-r}\brack{t-s}}_q\right).
\end{eqnarray*}
\end{theorem}
\begin{proof}
The first equation is a direct consequence of Equation \eqref{eq:sum ws is w} and Theorem \ref{thm:formula ws}. For the second equation, note that
$$\sum_{s=0}^r q^{\binom{s}{2}}\dfrac{[m]_q!}{[m-t]_q!}q^{s(\ell-r)}q^{\binom{t-s}{2}}
{{r}\brack{s}}_q {{\ell-r}\brack{t-s}}_q=\sum_{s=0}^r |\mathfrak{D}_t(\ell,m;r,s)|=|\mathfrak{D}_t(\ell,m)|,$$
since $\mathfrak{D}_t(\ell,m)$ is the disjoint union of the sets $\mathfrak{D}_t(\ell,m;r,s)$, $0 \le s \le r$.
\end{proof}

The above theorem in particular implies that
\begin{equation}\label{eq:ptr_alternative}
P_t(r)=\sum_{s=0}^r q^{\binom{s}{2}}(-1)^s \dfrac{[m-s]_q!}{[m-t]_q!}q^{s(\ell-r)}q^{\binom{t-s}{2}}
{{r}\brack{s}}_q {{\ell-r}\brack{t-s}}_q,
\end{equation}
where $P_t(r)$ is the expression from Equation \eqref{eq:ptr_delsarte}. It is not immediately clear that these two expressions for $P_t(r)$ are in fact equal. However, in \cite[Eq. (15)]{D2} the generalized Krawtchouk polynomial $F(x,k,n)$ is defined (involving parameters $x,k,n$ as well as a parameter $c$). If one chooses $c=q^{m-\ell}$, $n=\ell$, $k=t$, and $x=r$ one obtains the polynomial $P_t(r)$ from Equation \eqref{eq:ptr_delsarte}. A second and a third alternative expression for $F(x,k,n)$ are then given in \cite[Section 5.1]{D2}. The second one (with the same choice for the parameters $n,k,x$, and $c$ as before), precisely yields Equation \eqref{eq:ptr_alternative}.

\end{document}